\documentclass[leqno,11pt,a4paper]{amsart}
\usepackage{amsthm}
\usepackage{cite}
\usepackage{mathtools}
\usepackage{enumitem}
\usepackage{hyperref}
\usepackage{graphicx}
\usepackage{amssymb}
\usepackage{xcolor}
\hypersetup{
  colorlinks,
  linkcolor={blue!50!black},
  citecolor={blue!50!black},
  urlcolor={blue!80!black}
}
\usepackage{multirow}
\usepackage[all]{xy}
\usepackage{tikz}
\usepackage{booktabs}
\usepackage{arydshln}
\usepackage{mleftright}
\setlist[enumerate]{labelsep=*, leftmargin=1.5pc}
\setlist[enumerate]{label=\normalfont(\roman*), ref=\roman*}
\usepackage[margin=2.7cm]{geometry}
\graphicspath{{images/}}
\theoremstyle{plain}
\newtheorem{thm}{Theorem}[section]
\newtheorem{pro}[thm]{Proposition}
\newtheorem{lem}[thm]{Lemma}

\newtheorem{conjecture}[thm]{Conjecture}
\theoremstyle{definition}
\newtheorem{dfn}[thm]{Definition}
\newtheorem{rem}[thm]{Remark}
\newtheorem{eg}[thm]{Example}

\DeclareMathOperator{\Pic}{Pic}

\DeclareMathOperator{\Div}{Div}

\DeclareMathOperator{\Spec}{Spec}

\DeclareMathOperator{\Sing}{Sing}

\DeclareMathOperator{\Id}{Id}

\DeclareMathOperator{\Ann}{Ann}
\DeclareMathOperator{\cone}{cone}
\DeclareMathOperator{\TV}{TV}
\DeclareMathOperator{\im}{Im}

\newcommand{\fs}{\mathfrak{s}}

\newcommand{\cO}{\mathcal{O}}

\newcommand{\RR}{{\mathbb{R}}}
\newcommand{\PP}{{\mathbb{P}}}

\newcommand{\ZZ}{{\mathbb{Z}}}

\newcommand{\FF}{\mathbb{F}}

\newcommand{\yy}{\mathbf{y}}

\newcommand{\cM}{\mathbb{M}}
\newcommand{\tM}{\widetilde{M}}
\newcommand{\tN}{\widetilde{N}}

\renewcommand{\tilde}{\widetilde}
\newcommand{\conv}[1]{\operatorname{conv}\mleft({#1}\mright)}

\newcommand{\V}[1]{\operatorname{verts}\mleft({#1}\mright)}
\newcommand{\MM}[2]{{#1\text{--}#2}} % The Mori-Mukai number
\newcommand{\kk}{\Bbbk}

\begin{document}
%-------------------------------------------------------------------------------
\author[T.\,Prince]{Thomas Prince}
\address{Mathematical Institute\\University of Oxford\\Woodstock Road\\Oxford\\OX2 6GG\\UK}
\email{thomas.prince@magd.ox.ac.uk}% Prince

%-------------------------------------------------------------------------------
%-------------------------------------------------------------------------------
\title[Cracked Polytopes]{Cracked Polytopes and Fano Toric Complete Intersections}
\maketitle
%-------------------------------------------------------------------------------
\begin{abstract}
	We introduce the notion of \emph{cracked polytope}, and -- making use of joint work with Coates and Kasprzyk -- construct the associated toric variety $X$ as a subvariety of a smooth toric variety $Y$ under certain conditions. Restricting to the case in which this subvariety is a complete intersection, we present a sufficient condition for a smoothing of $X$ to exist inside $Y$. We exhibit a relative anti-canonical divisor for this smoothing of $X$, and show that the general member is simple normal crossings.
\end{abstract}

% !TEX root = paper.tex
%----------------------------------------------------------------------
\section{Introduction}
%----------------------------------------------------------------------

Given a Fano toric variety $X$ we consider the question, \emph{`when does $X$ smooth as a toric complete intersection?'} We provide a class of Fano toric varieties -- which we call \emph{cracked} -- for which this is possible under certain conditions.

Our smoothing of $X$ is constructed in two stages. First we use the method \emph{Laurent inversion} -- developed jointly with Coates and Kasprzyk in \cite{CKP17} -- to embed $X$ into a toric variety $Y$ such that $X$ degenerates to a union of toric strata of $Y$; and give conditions under which we can assume that $Y$ is smooth. Second, restricting to the complete intersection case, we study when the equations can be perturbed to smooth $X$. The result is a class of polytopes $P$, together with combinatorial conditions which ensure that the Fano toric variety $X_P$ can be smoothed.

The construction of log Calabi--Yau varieties by smoothing \emph{vertex varieties} has been studied in great depth by Gross, Hacking, Keel, and Siebert \cite{Gross--Siebert,GS06,GHS,GS16}; incorporating ideas from mirror symmetry, Gromov--Witten theory, logarithmic geometry and integral affine geometry. The smoothing we consider is related to a (special case of a) `compactified' analogue of this smoothing; where the vertex variety is now a collection of smooth projective toric varieties.

In \S\ref{sec:cracked_polytopes} we study the class of \emph{cracked} polytopes, polytopes which break into unimodular pieces, and their polar polytopes. In \S\ref{sec:laurent_inversion} we give a self-contained account of the procedure \emph{Laurent inversion} which takes as input a decoration of $P$ called a \emph{scaffolding}, see Definition~\ref{dfn:scaffolding}, and returns an embedding of $X_P$ into a toric variety $Y$. The content of \S\ref{sec:laurent_inversion} first appeared in the joint work~\cite{CKP17}. Our first main result is proved in \S\ref{sec:full_scaffolding}, which provides a characterisation of when $Y$ is smooth in a neighbourhood of the image of $X_P$. Note that given a fan $\Sigma$ we let $\bar{\Sigma}$ denote the quotient of $\Sigma$ by its minimal cone.

\begin{thm}
	\label{thm:smooth_ambient}
	Fix a polytope $P \subset M_\RR$, and a rational fan $\Sigma$ in $M_\RR$ such that the toric variety $Z := \TV(\bar{\Sigma})$ is smooth and projective. Given a scaffolding $S$ of $P$ with shape $Z$, we have that the target of the corresponding embedding is smooth in a neighbourhood of the image of $X_P$ if and only if $P$ is \emph{cracked} along $\Sigma$ and $S$ is \emph{full} -- see Definitions~\ref{dfn:cracked} and~\ref{dfn:full}.
\end{thm}

If $Z$ is a product of projective spaces it follows from Proposition~\ref{pro:ci_shape} that $X_P$ is a complete intersection in $Y_S$. If the line bundles defining this complete intersection are basepoint free, $X_P$ will smooth inside $Y_S$ by Bertini's theorem. We provide a weaker, though related, criterion for the smoothability of $X_P$ inside $Y_S$, which we call \emph{positivity} of $S$: a condition on the codimension one \emph{slabs} of the union of toric varieties to which $X_P$ degenerates. This criterion has practical and theoretical advantages over the na\"ive one. For example, in the context of the Gross--Siebert program it is related to the notion of positivity of the log structure on the central fibre of a toric degeneration, as defined in \cite{GS06}. We expect a precise understanding this connection to lead to a proof of \textbf{Conjecture~\ref{con:smoothing}}. We also explain in future work \cite{P:ZAffine} that this condition is closely related to the condition required to smooth a cracked polytope as an \emph{integral affine manifold}. In practice it also reduces the problem of computing basepoint loci on $Y_S$ to the problem of computing their restriction to $(\dim X_P - 1)$ dimensional toric strata of $Y_S$.

\begin{thm}
	\label{thm:affine_smoothing}
	Fix a fan $\Sigma$ such that $Z := \TV(\bar{\Sigma})$ is a product of projective spaces, and fix a polytope $P$ such that $P^\circ$ is cracked along $\Sigma$. Given a full scaffolding $S$ of $P$ with shape $Z$, the toric variety $X_P$ is the intersection of $r$ Cartier divisors corresponding to line bundles $L_1,\ldots, L_r \in \Pic Y_S$. If $S$ is \emph{positive} we have that:
	\begin{enumerate}
		\item\label{it:smoothing} The vanishing locus of a general section of $\bigoplus_{i \in [r]}{L_i}$ is a smooth variety.
		\item\label{it:divisor} There is a divisor $D_S$ on $Y_S$ such that  the restriction of $D_S$ to the vanishing locus of a general section of $\bigoplus_{i \in [r]}{L_i}$ is simple normal crossings and anti-canonical.
	\end{enumerate}
\end{thm}

While we do not attempt to classify cracked polytopes in this article, we study the special case in which (after removing torus factors) $Z \cong \PP^1$ in detail in \cite{CKP:P1}. In particular, among other results, we will give a classification when $\dim P \in \{3,4\}$.

The current work is related to the broad project of Coates, Corti, Galkin, Golyshev, Kasprzyk, and others to construct and classify Fano varieties via mirror symmetry \cite{CCGGK,CCGK,ACGK,A+}. In particular, given a scaffolding $S$ defining a complete intersection, there is a \emph{Laurent polynomial} $f_S$, defined in \cite{CKP17}. If $X_P$ is cut out by a collection of nef line bundles, the Quantum Lefschetz Hyperplane Theorem of Coates--Givental~\cite{CG07} implies that $f_S$ is mirror to the smoothing of $X_P$ in the sense defined in \cite{CCGGK}. Moreover \emph{positivity} of $S$ implies that $f_S$ admits a certain collection of \emph{mutations} -- that is, $f_S$ remains a regular function on a torus under certain birational transformations -- see \cite{Galkin--Usnich, ACGK, CCGGK} -- we will also return to this point in future work.

We also remark on a connection with polyhedral combinatorics. Cracked polytopes are reflexive polytopes made up of a number of \emph{hollow} polytopes, that is, of polytopes without interior points. These are themselves objects of interest and recent study, see for example \cite{AWW11,NZ11}. It would be interesting to investigate whether these works provide tools to allow us to classify cracked polytopes.

\subsection*{Acknowledgements}
We thank Tom Coates and Alexander Kasprzyk for many conversations about Laurent inversion. We also thank Alessio Corti and Andrea Petracci for many explanations and useful conversations. The author is supported by a Fellowship by Examination at Magdalen College, Oxford.

\subsection*{Conventions}
Throughout this article $N \cong \ZZ^n$ will refer to an $n$-dimensional lattice, and $M := \hom(N,\ZZ)$ will refer to the dual lattice. Given a ring $R$ we write $N_R := N \otimes_\ZZ R$ and $M_R := M \otimes_\ZZ R$. For brevity we let $[k]$ denote the set $\{1,\ldots,k\}$ for each $k \in \ZZ_{\geq 1}$. We work over an algebraically closed field $\kk$ of characteristic zero.
% !TEX root = paper.tex
%----------------------------------------------------------------------
\section{Cracked polytopes}
\label{sec:cracked_polytopes}
%----------------------------------------------------------------------

In this section we introduce the notion of \emph{cracked polytope}, which will form our main object of study, and, in the Fano setting, characterize its dual polytope. We will assume basic ideas and results from toric geometry -- see \cite{Cox--Little--Schenck,Fulton93} -- throughout.

\begin{dfn}
	\label{dfn:cracked}
	Fix a convex polyhedron $P \subset M_\RR$ containing the origin in its interior, and a unimodular fan $\Sigma$. We say $P$ is \emph{cracked along $\Sigma$} if every tangent cone of $P \cap C$ is unimodular for every maximal cone $C$ of $\Sigma$.
\end{dfn}

\begin{eg}
	Figure~\ref{fig:cracked_cubic} shows an example of a polygon $P$ cracked along the fan of $\PP^2$. The toric variety defined by the normal fan of $P$ is $\PP^2/\mu_3$. This surface is isomorphic to the singular cubic surface $\{x_1x_2x_3 = x_0^3\} \subset\PP^3$, where $x_i$, $i \in \{0,\ldots,3\}$ are homogeneous co-ordinates on $\PP^3$. Clearly, if we replace this binomial with a general cubic, the resulting variety is smooth. Here we see $P$ breaking into polytopes $P_i$, $i \in [3]$, which we will later identify with the toric divisors $\{x_i = 0\}$ for $i \in [3]$.

\begin{figure}
	\centering
	\begin{minipage}[b]{0.3\textwidth}
		\includegraphics[width=\textwidth]{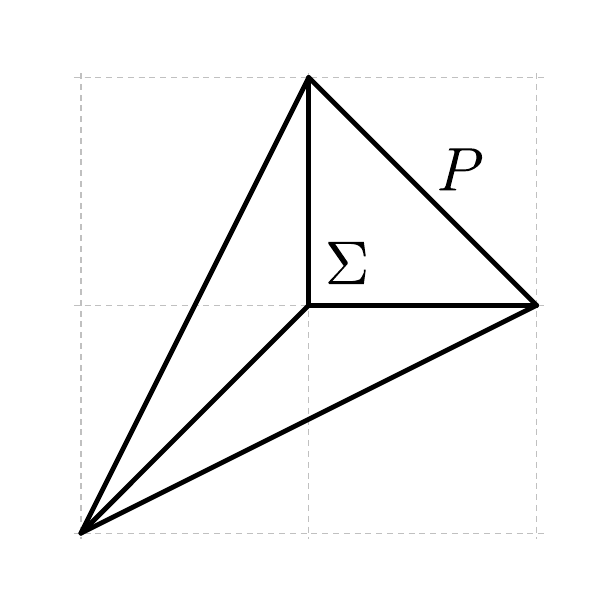}
	\end{minipage}
	\begin{minipage}[b]{0.3\textwidth}
		\includegraphics[width=\textwidth]{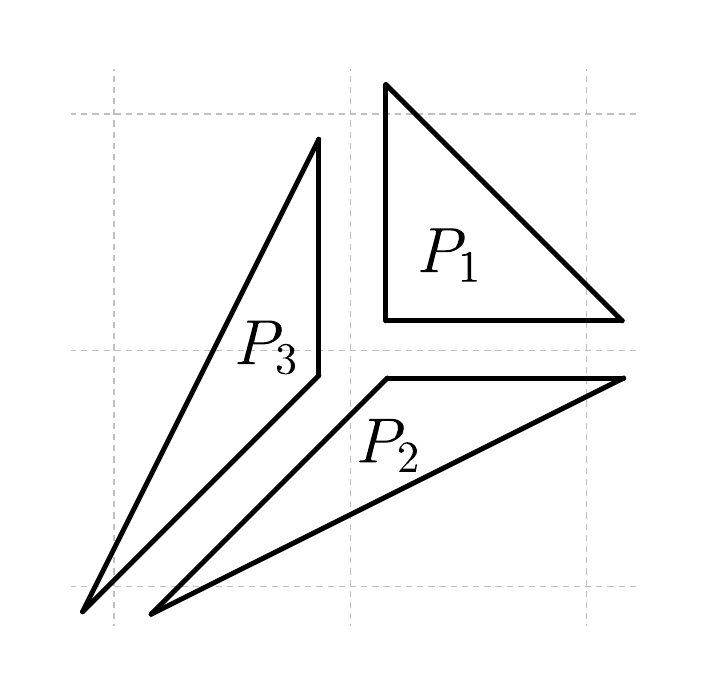}
	\end{minipage}
	\caption{Cracking a polygon.}
	\label{fig:cracked_cubic}
\end{figure}
 
\end{eg}

\begin{eg}
	We present two three-dimensional examples of cracked polytopes in Figure~\ref{fig:cracked_polytopes}. The left-hand example uses the simplest non-trivial fan, consisting of two maximal cones meeting along a hyperplane. The normal fan of the toric variety shown is isomorphic to the blow-up in a smooth point of a quadric in $\PP^4$ which contains a line of singularities. This variety admits an embedded smoothing in  of $\PP^4$ blown up at a (reduced) point.
	
	The polytope $P$ shown in the right-hand example is cracked along the fan of $\PP^1\times \PP^1 \times \PP^1$. The normal fan of $P$ in this case defines a toric variety isomorphic to
	\[
		V(x_1y_1 - x_0y_0, x_2y_2 - x_0y_0, x_3y_3 - x_0y_0) \subset \PP^3\times\PP^3
	\]
	where $x_i$ and $y_i$ are homogeneous co-ordinates on the respective $\PP^3$ factors. $X_P$ contains $12$ ordinary double point singularities. These are smoothed by perturbing the equations defining $X_P$ in $\PP^3 \times \PP^3$, which then define the vanishing locus of a general section of $\bigoplus_{i \in [3]}\cO_{\PP^3\times \PP^3}(2,2)$.
	\begin{figure}
		\centering
		\begin{minipage}[b]{0.49\textwidth}
			\includegraphics[width=\textwidth]{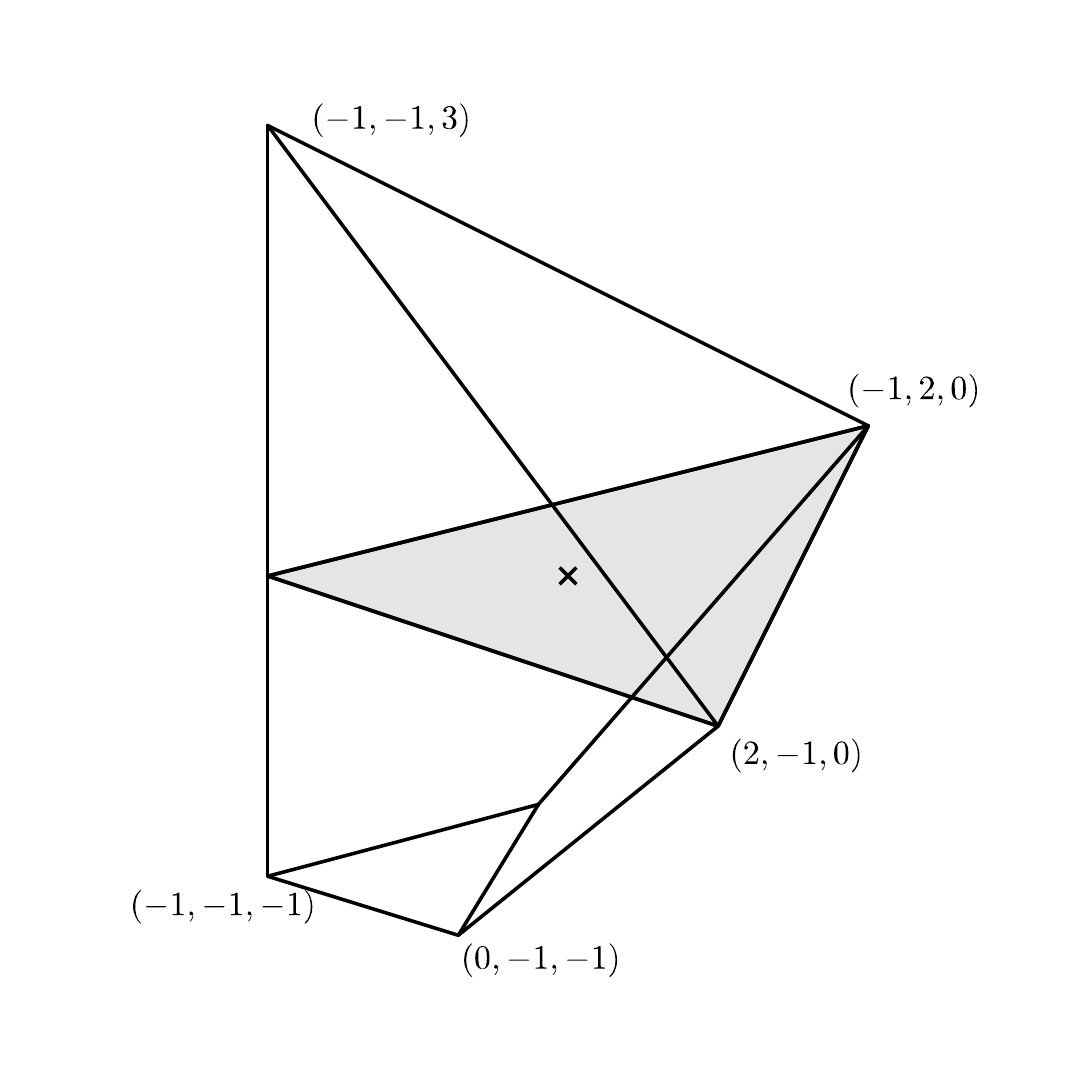}
		\end{minipage}
		\hfill
		\begin{minipage}[b]{0.49\textwidth}
			\includegraphics[width=\textwidth]{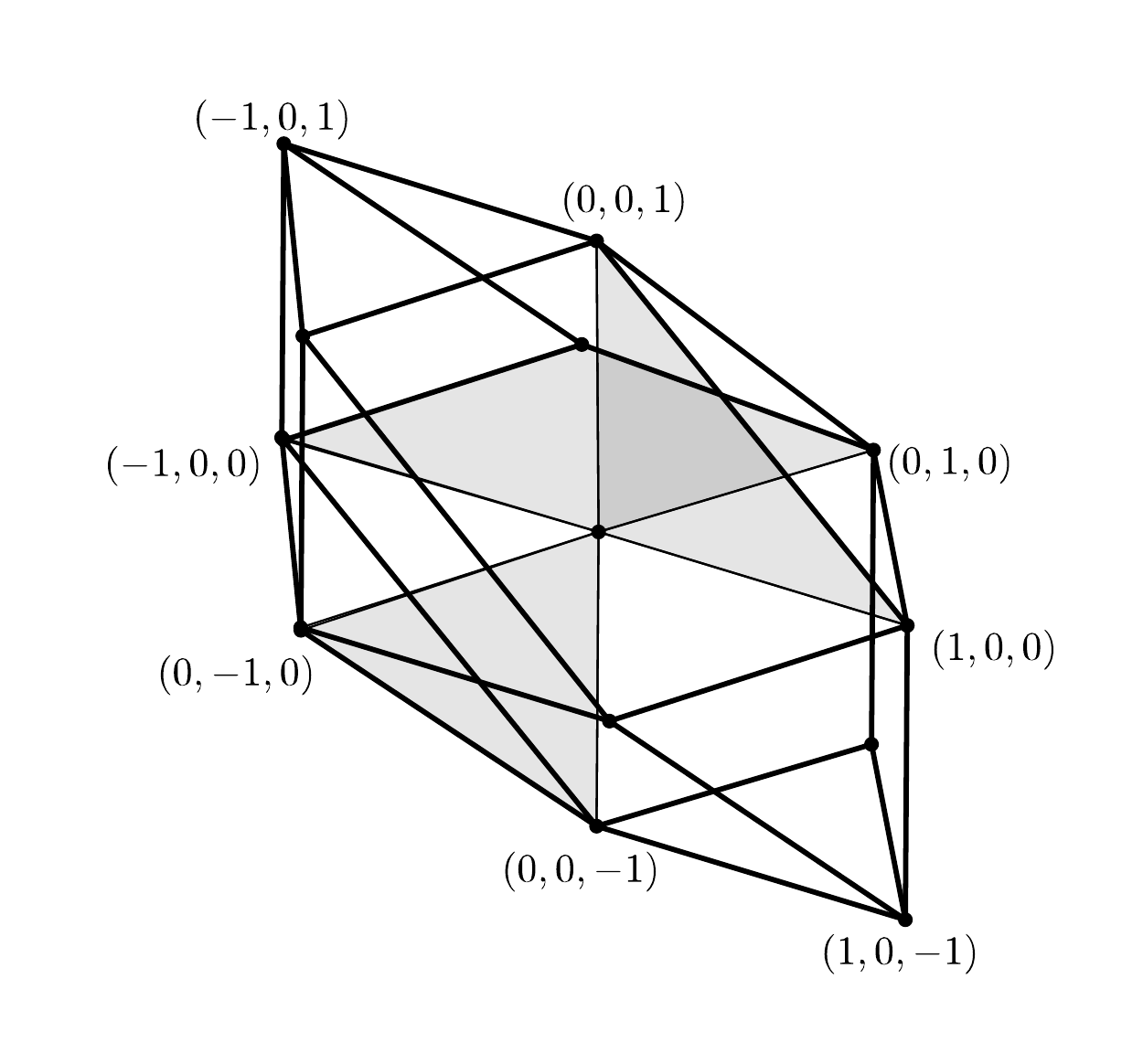}
		\end{minipage}
		\caption{Examples of cracked polytopes.}
		\label{fig:cracked_polytopes}
	\end{figure}
\end{eg}

Note that we do not assume the origin is the minimal cone of $\Sigma$ in Definition~\ref{dfn:cracked}. We let $\bar{M}$ denote the quotient of $M$ by the minimal cone of $\Sigma$.

\begin{rem}
	Although we provide a general definition, this article is concerned solely with the case that $\bar{\Sigma}$ defines a \emph{projective} toric variety, and $P$ is a polytope containing the origin in its interior.
\end{rem}

We study cracked polyhedra which are images of moment maps of an anti-canonically polarised toric Fano varieties. Recall that an integral polytope $P \subset N_\RR$ is called \emph{Fano} if it contains the origin in its interior, and every vertex of $P$ is primitive.

In general the polar polytope $P^\circ \subset M_\RR$ of a Fano polytope $P$ is not integral, which is true in the special case that $P$ is \emph{reflexive}. Given a Fano polytope $P$ its \emph{spanning fan} is the fan whose cones are given by the cones over faces of $P$, and we denote the toric variety determined by the spanning fan of $P$ by $X_P$. We also recall that there is an inclusion reversing map between the faces of $P$ and $P^\circ$. Let $F^\star$ denote the face dual to the face $F$ of $P$ or $P^\circ$.

\begin{pro}
	\label{pro:cracked_is_reflexive}
	Let $P$ be a Fano polytope cracked along a unimodular and complete fan $\Sigma$, then $P$ is reflexive.
\end{pro}
\begin{proof}
	We show that, for every vertex $v \in \V{P^\circ}$, the vertices of the facet $v^\star$ of $P$ lie in a hyperplane of $\{ v \in N_\RR : \langle u,v\rangle = 1\}$ for some $u \in M$. Let $C$ be the minimal cone of $\Sigma$ containing $v$ in its relative interior, and let $k := \dim C$. Let $C_v$ denote the tangent cone of $P^\circ$ at $v$, and let $\{b_1,\ldots, b_k\}$ be the minimal generating set of the unimodular cone $C_v\cap \langle C\rangle$.

	Every facet $F$ of $P^\circ$ containing $v$ is contained in an affine hyperplane spanned by $n-1$ of the vectors $\{b_1,\ldots, b_n\}$, the basis of the tangent cone to $C_v \cap \langle B \rangle$ at $v$ for some maximal cone $B$ of $\Sigma$. Moreover this collection cannot contain $\{b_1,\ldots, b_k\}$, as the affine subspace spanned by these $k$ vectors contains the origin. Thus if $F^\star = \{w\}$, $w = b_j^\star$ for some $j \in [k]$ and $\langle \sum_{i \in [k]}{b_i}, w\rangle = 1$. Since $\sum_{i \in [k]}{b_i}$ is defined independently of $F$ and $B$, $v^\star \subset \langle \sum_{i \in [k]}{b_i}, -\rangle = 1$.
\end{proof}

Let $P \subset N_\RR$ be a reflexive polytope such that $P^\circ$ is cracked along a fan $\Sigma$ in $M_\RR$; assume moreover that $\Sigma$ defines a projective toric variety $Z$. We now characterize the facets of $P$.

\begin{dfn}
	Recall the \emph{Cayley sum} $P_1 \star \cdots \star P_r$ of polytopes $P_i \subset N_\RR$ for $i \in [r]$ is the convex hull of the union of the polytopes $P_i + e_i$ in $N_\RR \oplus \RR^r$ for $i \in [r]$.
\end{dfn}

Given a fan $\Sigma$ in $M_\RR$ defining a projective toric variety $Z$, and $C$ a cone of $\Sigma$, let $Z_C$ denote the subvariety of the toric boundary corresponding to $C$ under the orbit-cone correspondence.
\begin{dfn}
	Let $P \subset N_\RR$ be a Fano polytope and $\Sigma$ a fan in $M_\RR$. We say \emph{$P$ has facets of Cayley type} if every facet $F$ of $P$ is affine linearly isomorphic to the Cayley sum of polyhedra associated to nef divisors of $Z_C$, where $C$ is the minimal cone of $\Sigma$ containing the vertex $F^\star$ of $P^\circ$. Moreover we insist that that isomorphism identifies $\Ann\langle C \rangle\subset N$ with the character lattice of $Z_C$.
\end{dfn}

\begin{pro}
	\label{pro:cayley_sum}
	Fix a reflexive polytope $P$ such that $P^\circ$ is cracked along a fan $\Sigma$. Assuming that $Z := \TV(\bar{\Sigma})$ is a smooth projective toric variety, $P$ has facets of Cayley type.
\end{pro}
\begin{proof}
	Let $C_v$ denote the tangent cone to $P^\circ$ at $v$. Let $\Sigma_v$ the fan induced by $C_v$ in the quotient space $M_\RR/\langle v\rangle$. Let $C$ be the minimal cone of $\Sigma$ containing $v$, and let $\{b_1,\ldots, b_k\}$ denote the minimal generating set of $C_v \cap \langle C\rangle$, where $k := \dim C$.
	
	We construct a fan  $\tilde{\Sigma}_v$ in $M_\RR/\langle v \rangle$ refining $\Sigma_v$. Fix a maximal cone $\sigma \ni v$ of $\Sigma$ and let $\{b_1,\ldots,b_n\}$ be the extension of $\{b_1,\ldots, b_k\}$ to the minimal generating set of the tangent cone $\sigma_v$ to $\sigma$ at $v$. The cone $\langle b_i : i \in [n], i\neq j \rangle$ projects to a full-dimensional unimodular cone in $M_\RR/\langle v \rangle$ for any $j \in [k]$. We define $\tilde{\Sigma}_v$ to be the complete fan formed by these maximal cones for all $\sigma \ni v$.

	As every maximal cone of $\tilde{\Sigma}_v$ contains $k-1$ vectors $b_i$ for $i \in [k]$, the toric variety $X_v$ defined by $\tilde{\Sigma}_v$ admits a projection given by the quotient $M_\RR/\langle v \rangle \to M_\RR/\langle C \rangle$ to the toric variety $Z_C$. Since, by the proof of Proposition~\ref{pro:cracked_is_reflexive}, the anti-canonical (Gorenstein) direction for $\cone(v^\star)$ is $u := \sum_{i \in [k]}{b_i}$, $u$ is in the kernel of the map $M_\RR \to M_\RR/\langle v \rangle$. Thus the fibres of the projection $X_v \to Z_C$ are isomorphic to $\PP^{k-1}$.
	
	We write $C_v$ as the region above a convex PL-function $\theta \colon M_\RR/\langle v\rangle \to \RR$. First note that, by reflexivity of $P$, fixing a vertex $w$ of the facet $v^\star$ determines a splitting of $N$ into the sublattice annihilating $v \in M$, and the direction generated by $w$. Identifying $M_\RR/\langle v \rangle$ with $\Ann(w)$ we identify $M_\RR$ with $M_\RR/\langle v \rangle \oplus \langle v \rangle$, and hence identify $\partial C_v$ with a graph of a function $\theta$ on $M_\RR/\langle v \rangle$. Note that $v^\star$ is the polyhedron of sections of the corresponding divisor on the toric variety $X_v$ associated to $\tilde{\Sigma}_v$.
	
	Let $\bar{b}_i$ be the images of $b_i$ in $M_\RR/\langle v \rangle$ for $i \in [k]$. Recall that each maximal cone of $\tilde{\Sigma}_v$ contains ${k-1}$ of the vectors $\bar{b}_i$, for $i \in[k]$. Therefore the function $\theta$ vanishes on all but a single ray of the $k$ rays in the subspace $\langle C\rangle/\langle v\rangle \subset M_\RR/\langle v \rangle$, and evaluates to $1$ on the remaining ray by unimodularity. Relabelling the elements $b_i$ we may assume that $\langle u, b_i\rangle = 0$ for all $i \in [k-1]$. Since $\theta(\bar{b}_k) = 1$, $v^\star$ (regarded as a polytope in $\Ann(v)$) is contained in the column
	\[
	\bigcap_{i \in [k-1]} \{u : u \in \Ann(v), \langle \bar{b}_i,u\rangle\geq 0\} \cap \{u : u \in \Ann(v), \langle \bar{b}_k, u \rangle \geq -1\}.
	\]
	Thus $v^\star$ projects to the standard simplex in the vector space dual to $\langle C\rangle/\langle v\rangle$ and hence is a Cayley sum.
	
	Recall that there is a \emph{surjection} from the set of maximal cones of $\tilde{\Sigma}_v$ to the vertices of $v^\star$. Fix a maximal cone of $\tilde{\Sigma}_v$ and assume, without loss of generality, that it contains the vectors $\{\bar{b}_1,\ldots, \bar{b}_{k-1}\}$.  The vertex of $v^\star$ dual to this maximal cone is contained in the subspace annihilating every $\bar{b}_i$ for $i \in [k-1]$. The face of $v^\star$ contained in this subspace, is nothing other than the polyhedron of sections of the divisor on $Z_C$ obtained from the convex piecewise linear function induced by $\theta$ on the quotient fan $\tilde{\Sigma}_v/\langle\bar{b}_1,\ldots,\bar{b}_{k-1}\rangle$. This is a nef divisor on $Z_C$ as $\theta$ is a convex function on the fan $\tilde{\Sigma}_v$.
\end{proof}

\begin{rem}
	Note that the converse to Proposition~\ref{pro:cayley_sum} is not true. For example, any smooth lattice polytope $P \subset N_\RR$ has Cayley facets for the fan subdividing $M_\RR$ into two half-spaces meeting along the annihilator of any vector $u \in N$ such that $v \notin \Ann(u)$ for all $v \in \V{P^\circ}$. However slicing $P^\circ$ by the annihilator of $u$ will not produce a pair of polytopes with unimodular tangent cones in general.
\end{rem}

Observe that any facet $F$ of a Fano polytope $P \subset N_\RR$ such that $P^\circ$ is cracked along a fan $\Sigma$ admits a projection $\pi_F \colon F \to \Delta_{k-1}$ where $k$ is the dimension of the minimal cone of $\Sigma$ containing the dual vertex to the facet $F$, and $\Delta_l$ is the standard simplex of dimension $l$.

\begin{dfn}
	Given a reflexive polytope $P$ such that $P^\circ$ is cracked along a fan $\Sigma$ we say a face $E$ is \emph{vertical} if, for any facet $F$ containing $E$, $\pi_F(E)$ is a vertex of $\Delta_{k-1}$,
\end{dfn}
% !TEX root = paper.tex
%----------------------------------------------------------------------
\section{Laurent Inversion}
\label{sec:laurent_inversion}
%----------------------------------------------------------------------

In this section we recall the method, Laurent inversion, developed in \cite{CKP17}. Throughout this section we fix a lattice $N \cong \ZZ^n$, a splitting of $N = \bar{N} \oplus N_U$ and a Fano polytope $P \subset N_\RR$. Given these data we can define the notion of \emph{scaffolding} on $P$, see \cite[Definition~$3.1$]{CKP17}.

\begin{dfn}
	\label{dfn:scaffolding}
	Fix a smooth projective toric variety $Z$ with character lattice $\bar{N}$. A \emph{scaffolding of $P$} is a set of pairs $(D,\chi)$ where $D$ is a nef divisor on $Z$ and $\chi$ is an element of $N_U$, such that
	\[
	P = \conv{P_D + \chi \, \Big| \, (D,\chi) \in S}.
	\]
	We refer to $Z$ as the \emph{shape} of the scaffolding, and elements $(D,\chi) \in S$ as \emph{struts}. We also assume that  for every vertex of $P$, there is a unique $s = (D,\chi)$ such that $v \in P_D+\chi$.
\end{dfn}

Note that the assumption that vertices meet a unique polytope $P_D+\chi$ did not appear in the Definition given in \cite{CKP17}. This is an innocuous technical condition which we make use of to prove Theorem~\ref{thm:smooth_ambient}.

We let $\ell$ denote the rank of the free abelian group $\Div_{T_{\bar{M}}}Z$. Following \cite{CKP17} we show that a scaffolding yields a torus invariant embedding of $X_P$ into an ambient toric variety $Y_S$.

\begin{dfn}[{\!\!\cite[Definition~$A.1$]{CKP17}}]
	\label{dfn:ambient}
	Given a scaffolding $S$ of $P$ we define a toric variety $Y_S$, the normal fan of the polytope $Q_S \subset \tilde{M}_\RR := (\Div_{T_{\bar{M}}}Z \oplus M_U)\otimes_\ZZ \RR$, itself defined by the following inequalities:
	\[
	\begin{cases}
	\big\langle (-D,\chi), - \big\rangle \geq -1 & \textrm{for all $(D,\chi) \in S$};\\
	\big\langle (0,e_i), - \big\rangle \geq 0 & \textrm{for $i \in [\ell]$},
	\end{cases}
	\]
	where $e_i$ denotes the standard basis of $\Div_{T_{\bar{M}}}Z \cong \ZZ^\ell$. 
\end{dfn}

We let $\Sigma_S$ denote the normal fan of the polytope $Q_S$, and let $E_i$ denote the divisor of $Z$ corresponding to the lattice vector $e_i$. We also define $\rho_s := (-D,\chi)$ for each $s = (D,\chi) \in S$. We define a map of lattices
\[
\xymatrix@R-1pc{
 \llap{$\theta := \rho^\star\oplus \Id \colon $} \bar{N} \oplus N_U \ar[rr]& & \Div_{T_{\bar{M}}}(Z) \oplus N_U, \\
 N \ar@{=}[u] & & \tilde{N} \ar@{=}[u]
}
\]
where $\rho$ is the ray map of the fan $\bar{\Sigma}$ determined by $Z$.

\begin{thm}[{\!\!\cite[Theorem~$5.5$]{CKP17}}]
	\label{thm:embedding}
	A scaffolding $S$ of a polytope $P$ determines a toric variety $Y_S$ and an embedding $X_P \to Y_S$. This map is induced by the map $\theta$ on the corresponding lattices of one-parameter subgroups.
\end{thm}

Let $\V{S}$ denote the set of torus fixed points of $Z$, and, for each $u \in \V{S}$, let $C_u$ denote the intersection of the maximal cone of $\Sigma$ corresponding to $u$ with $P^\circ$. Observe that, given a nef divisor $D$ on $Z$, there is a canonical surjection $\V{S} \rightarrow \V{P_{D}}$. We denote this map $v \mapsto v^D$. Each element $u \in \V{S}$ defines a function $u \colon S \to N$, defined by setting $u\left((D,\chi)\right) = u^D + \chi$.

\begin{dfn}
	\label{dfn:iota}
	Let $\iota$ be the inverse map to the restriction to $\Gamma \oplus N_U$ of the canonical projection $\tM_\RR \rightarrow M_\RR$,  where $\Gamma$ is the union of $(n - \dim N_U)$-dimensional faces of the standard coordinate cone in $\Div_{T_{\bar{M}}}(Z)^\star$ which project onto maximal dimensional cones of $\bar{\Sigma}$.
\end{dfn}

Let $\iota_u \colon M_\RR \to \widetilde{M}_\RR$ be the linear extension of the map $\iota|_{C_u}\colon C_u \to \widetilde{M}_\RR$ for each $u \in \V{S}$.

\begin{lem}
	\label{lem:projecting_struts}
	Given an element $s \in S$ and $u \in \V{S}$, we have that
	\[
	\iota_u^\star\rho_s = u(s).
	\]
\end{lem}
\begin{proof}
	The ray generators of the maximal cone in $M_\RR$ corresponding to $u$ form a basis $\{\bar{e}_i : i \in [\dim(\bar{M})]\}$ of $\bar{M}$. Moreover the vectors $\iota_u(\bar{e}_i)$ are standard basis vectors $e^\star_i$ in $\Div_{T_{\bar{M}}}(Z)^\star \subseteq \tM_\RR$. Thus we have that
	\[
		\langle \iota_u^\star\rho_s, \bar{e}_i\rangle = \langle \rho_s, e^\star_i\rangle.
	\]
	Writing $s = (D,\chi)$, one of the defining inequalities of $P_D$ is 
	\[
		\langle -,\bar{e}_i\rangle \geq -\langle \rho_s, e^\star_i\rangle.
	\] 
	That is, writing the projection of $\iota_u^\star\rho_s$ to $\bar{N}$ in co-ordinates determined by the basis $\bar{e}^\star_i$, and recalling that $\rho_s = (-D,\chi)$, we have that these co-ordinates are identical to those of $u(s)$. Note that since $\iota_u$ acts as the identity on $M_U$ the result follows.
\end{proof}

\begin{pro}[{\!\!\cite[Proposition~A.$9$]{CKP17}}]
	\label{pro:iota_union}
		The polytope $\iota(C_u)$ is a face of $Q_S$ for each $u \in \V{S}$.
\end{pro}
\begin{proof}
	The polytope $\iota(C_u)$ is clearly contained in the boundary of the standard positive cone. Given any $s \in S$ and $p \in C_u$, $\langle \rho_s,\iota(p)\rangle = \langle u(s), p\rangle \geq -1$ by Lemma~\ref{lem:projecting_struts}. Thus $\iota(C_u)$ is contained in a face of $Q_S$; the reverse inclusion follows similarly.
\end{proof}

\begin{lem}
	\label{lem:normals}
	Given a vertex $v \in \V{P^\circ}$, the tangent cone of $Q_S$ at $\iota(v)$ is defined by the following inequalities:
	\[
	\begin{cases}
	\big\langle \rho_s, - \big\rangle \geq -1 & \textrm{$s=(D,\chi) \in S$ such that $(P_D+\chi) \cap v^\star \neq \varnothing$};\\
	\big\langle (e_i,0), - \big\rangle \geq 0 & \textrm{$u \notin E_i$ for some $u$ such that $v \in C_u$},
	\end{cases}
	\]
\end{lem}
\begin{proof}
	By Lemma~\ref{lem:projecting_struts} $\langle \rho_s, \iota(v)\rangle = \langle u(s),v \rangle$ for any $u \in \V{S}$ such that $v \in C_u$. This is equal to $-1$ if and only if $u(s) \in v^\star$.
	The second set inequalities follow as $\iota(v)$ is in the span of those $e^\star_i$ corresponding to rays of $\bar{C}$, where $C$ is the minimal cone of $\Sigma$ containing $v$ and $\bar{C}$ is the projection of $C$ to $\bar{M}$.
\end{proof}

\begin{proof}[Proof of Theorem~\ref{thm:embedding}]
	Given a vertex $v \in \V{P^\circ}$, let $C_v$ denote the tangent cone of $P^\circ$ at $v$, and let $\tilde{C}_v$ denote the tangent cone of $Q_S$ at $\iota(v)$. We prove that $\theta^\star(\tilde{C}_v) = C_v$. By Proposition~\ref{pro:iota_union} we have that $C_v \subseteq \theta^\star(\tilde{C}_v)$. Fix a point $p \in \tilde{C}_v$ and a vertex $w \in \V{v^\star}$. We have that $w = \iota_u^\star\rho_s$ for some $s \in S$ and $u \in \V{S}$. Now $\langle \theta(w),p \rangle = \langle \rho_s,p \rangle + \langle \theta(w)-\rho_s,p \rangle$. Note that $\langle \rho_s,p \rangle \geq -1$ by Lemma~\ref{lem:normals}. After projecting $\tN \to \Div_{T_{\bar{M}}}Z$, the polyhedron of sections of the divisor $\theta(w)-\rho_s$ is the translate of $P_D$ defined by taking the vertex $w$ to the origin. Thus, writing out $\theta(w)-\rho_s$ in the basis $e_i$, $i \in [\ell]$, the components corresponding to divisors $E_i$ containing any $u$ such that $u(s) = w$ vanish; while all others have non-negative coefficient. Thus $\langle \theta(w)-\rho_s,p \rangle\geq 0$, and $\langle w,\theta^\star(p)\rangle \geq -1$, as required. Finally, we need to show that the map $\theta^\star$ defines a surjection of semigroups. This follows from Proposition~\ref{pro:iota_union}: as $Z$ is smooth each $\iota_u$ is an integral splitting of $\theta^\star$.
\end{proof}

We describe the construction of $Y_S$ and the embedding $X_P \hookrightarrow Y_S$ in a simple example. 

\begin{eg}
	\label{eg:running}
	Consider the polygon $P$ shown in Figure~\ref{fig:first_scaffolding}. Fixing the shape variety $\PP^2$, Figure~\ref{fig:first_scaffolding} shows a scaffolding of $P$ with two struts. The dual polytope is shown on the right-hand side of Figure~\ref{fig:first_scaffolding}, where it is easily seen that this polygon is cracked along the fan of $\PP^2$. Note that in this example $\overline{N} = N$ and $N_U = \{0\}$.
	
	\begin{figure}
		\centering
		\begin{minipage}[b]{0.2\textwidth}
			\includegraphics[width=\textwidth]{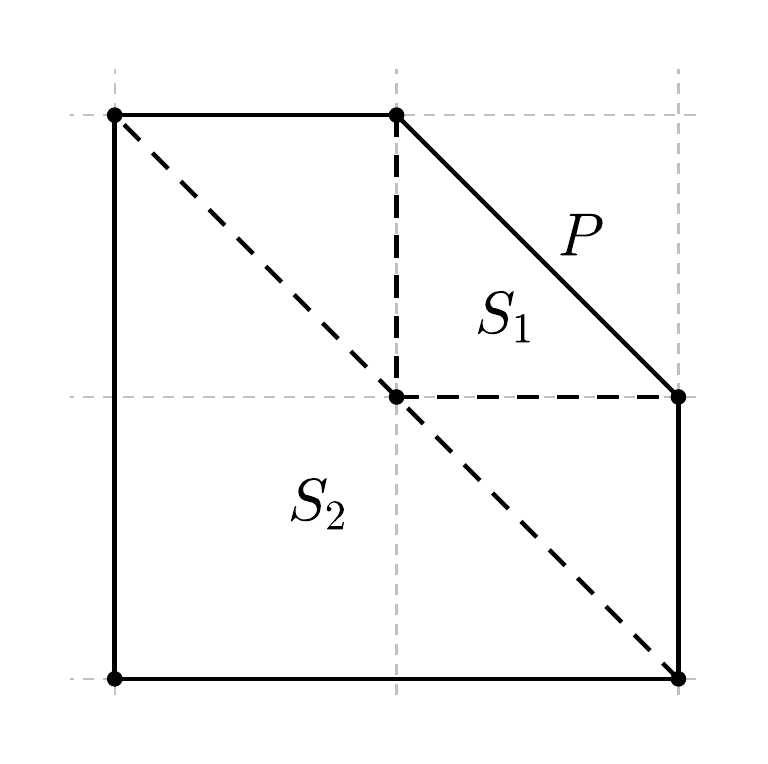}
		\end{minipage}
		%\hfill
		\begin{minipage}[b]{0.2\textwidth}
			\includegraphics[width=\textwidth]{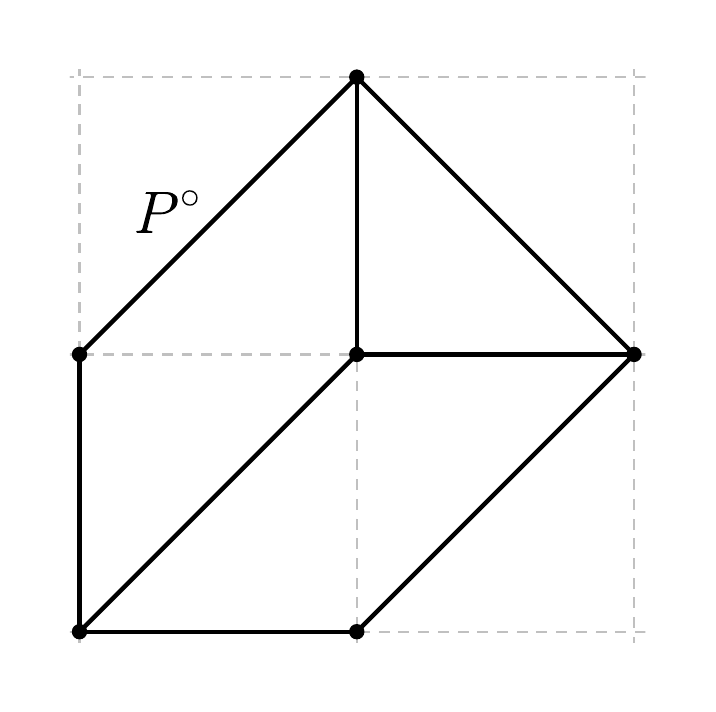}
		\end{minipage}
		\caption{Scaffolding a polygon.}
		\label{fig:first_scaffolding}
	\end{figure}

	The polytope $Q_S \subset \RR^3$ is defined by the inequalities $\langle e_i, -\rangle \geq 0$ for all $i \in [3]$, together with the two additional inequalities. These are obtained from the two divisors on $\PP^2$ whose polyhedra of sections are shown in Figure~\ref{fig:first_scaffolding}. Identifying the standard basis $e_i$ of $\ZZ^3 \cong \Div_{T_{\bar{M}}}(\PP^2)$ with specific divisors of $\PP^2$ we add the inequalities $\langle -e_1-e_2, -\rangle \geq -1$ and $\langle -e_3, - \rangle \geq -1$ to those defining $Q_S$. We display the map $\iota$ in Figure~\ref{fig:iota_eg}.	Note that $\TV(\Sigma_S) \cong \PP^2 \times \PP^1$ and the image of $X_P$ is a hypersurface defined by a section of $L := \cO_{\PP^2}(2) \boxtimes\cO_{\PP^1}(1)$. It is well known, for example by projecting to the $\PP^2$ factor, that general members of the linear system defined by $L$ are smooth del~Pezzo surfaces of degree $5$.
	
	\begin{figure}
		\includegraphics[scale=0.5]{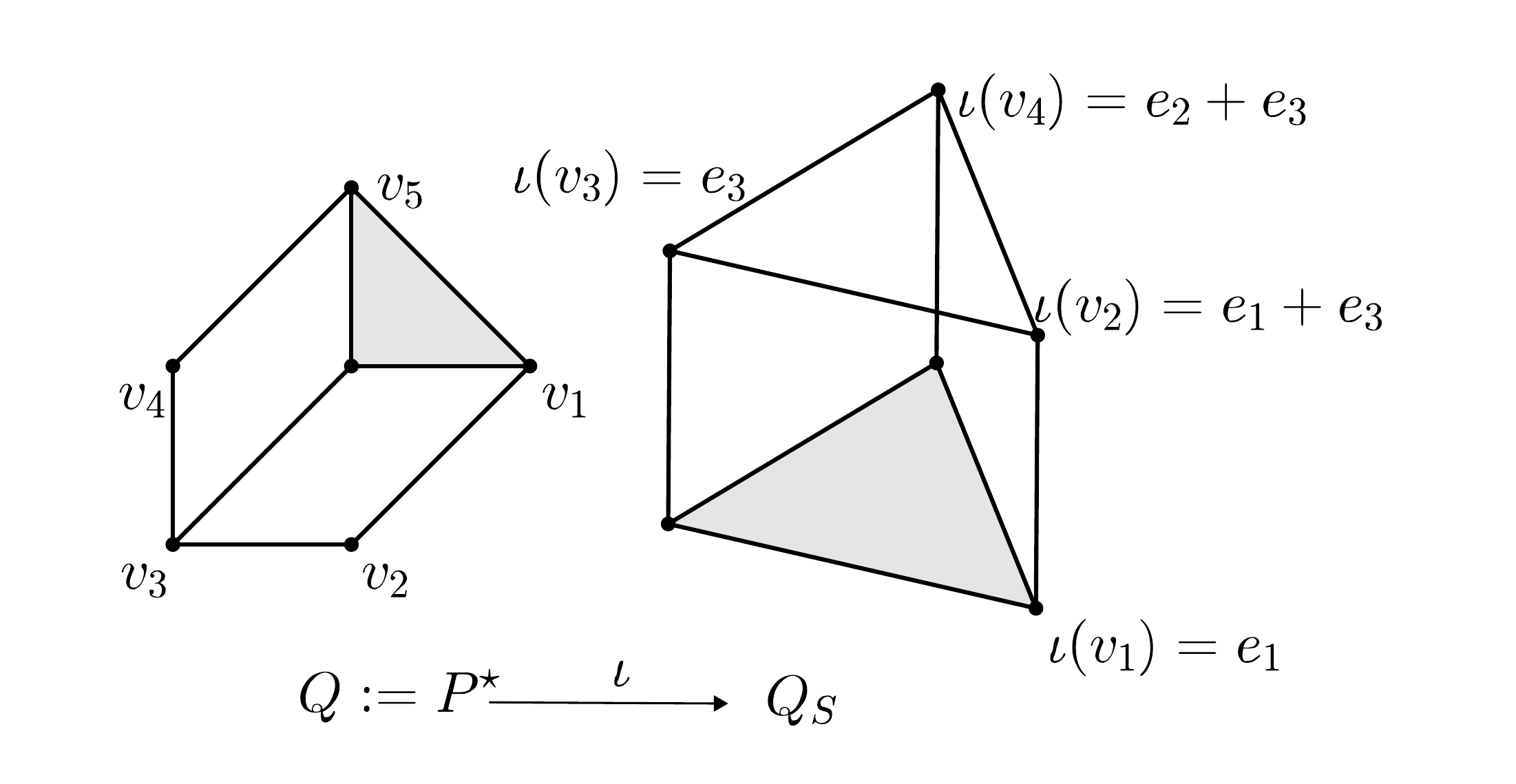}
		\caption{Example of the piecewise linear map $\iota$.}
		\label{fig:iota_eg}
	\end{figure}
\end{eg}
% !TEX root = paper.tex
%----------------------------------------------------------------------
\section{Full Scaffoldings}
\label{sec:full_scaffolding}
%----------------------------------------------------------------------

In this section we introduce the notion of \emph{full} scaffolding, and complete the proof of Theorem~\ref{thm:smooth_ambient}. If we fix a Fano polytope $P$, there are a vast number of possible scaffoldings, for many different shape varieties $Z$. We will control this class in two ways: first we constrain the class of polytopes to those which are cracked along the fan determined by $Z$; second we insist that our scaffoldings are \emph{full}, which often uniquely determines a scaffolding on $P$ with a given shape.

\begin{dfn}
	\label{dfn:full}
	Given a Fano polytope $P \subset N_\RR$ cracked along a fan $\Sigma$ in $M_\RR$ we say a scaffolding $S$ of $P$ with shape $Z := \TV(\bar{\Sigma})$ is \emph{full} if every vertical face of $P$ is contained in a polytope $P_D+\chi$ for a (unique) element $(D,\chi) \in S$.
\end{dfn}

Unfortunately full scaffoldings of a cracked polytope need not exist, and if they do exist, they need not be unique.

\begin{eg}
	Figure~\ref{fig:no_full} shows a polygon $P$, which we attempt to scaffolding using the shape $Z = \PP^2$. We show the dual polygon on the right hand side of Figure~\ref{fig:no_full}, from which we can easily see that $P$ is cracked along the fan $\Sigma$ determined by $Z$. However the three vertical faces (the three edges of $P$ whose normal directions are rays of $\Sigma$) of $P$ cannot be covered by a single strut of a scaffolding, as shown.

	\begin{figure}
		\centering
		\begin{minipage}[b]{0.2\textwidth}
			\includegraphics[width=\textwidth]{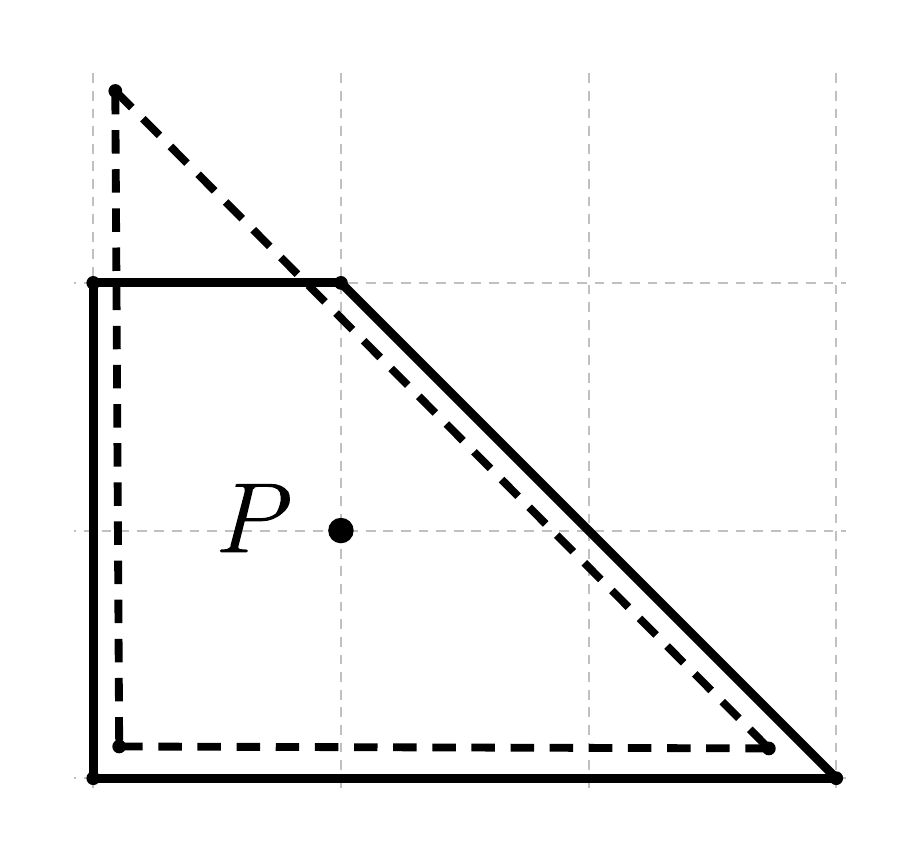}
		\end{minipage}
		%\hfill
		\begin{minipage}[b]{0.2\textwidth}
			\includegraphics[width=\textwidth]{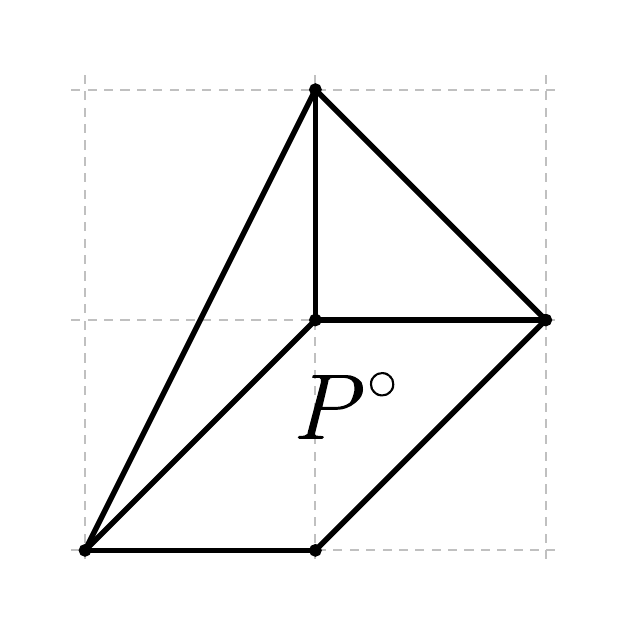}
		\end{minipage}
		\caption{Cracked polytope with no full scaffolding.}
		\label{fig:no_full}
	\end{figure}
	
\end{eg}

\begin{rem}
	\label{rem:P1-shape}
	Note that in the case that the shape variety of $Z$ is one-dimensional (that is, $Z$ is isomorphic to $\PP^1$) full scaffoldings for polytopes cracked along the fan defined by $Z$ always exist. Indeed, vertical faces of $P$ are precisely edges in the direction annihilating the minimal cone of $\Sigma$. Scaffold $P$ by covering each such (vertical) edge with a single polytope $P_D+\chi$, and add any struts to cover each remaining vertex of $P$.
\end{rem}

\begin{proof}[Proof of Theorem~\ref{thm:smooth_ambient}]
We first show that if $P$ is cracked along $\Sigma$, and admits a full scaffolding with shape $Z = \TV(\bar{\Sigma})$, then the tangent cone $\tilde{C}_v$ to $Q_S$ at $\iota(v)$ is a unimodular cone for any $v \in \V{P^\circ}$.

Let $C$ be the minimal cone of $\Sigma$ containing $v$. We let $k := \dim C$, and $n_u := \dim N_U$. We count the inequalities defining $\tilde{C}_v$ given in Lemma~\ref{lem:normals}. The description of $v^\star$ as a Cayley sum implies that there are $k$ inequalities of the form $\langle \rho_s, - \rangle \geq -1$. Let $\bar{C}$ be the image of $C$ in $\bar{\Sigma}$ and note that, since $Z$ is smooth, there are $(k-n_u)$ rays of $\bar{C}$. Any ray generator of $\bar{\Sigma}$ which is not a ray of $C$ defines an inequality $\langle (e_i,0), - \rangle \geq 0$ appearing in Lemma~\ref{lem:normals}. Thus the total number of inequalities defining $\tilde{C}_v$ is $(\ell - k + n_u) + k = \dim \tilde{N}$, hence $\tilde{C}_v$ is simplicial. Let $d:= \ell - k + n_u$.

Fixing a basis of $N_U$ we record the ray generators of $\tilde{C}_v^\star$ in the columns of the matrix
\[
\cM = 
\begin{pmatrix}
I_d & A \\
0 & B.
\end{pmatrix}
\]

We need to check that $\det(B) = \pm 1$. The last $k$ columns of $\cM$ are the vectors $\rho_s$ such that $u(s) \in v^\star$ for some $u \in \V{S}$. The last $k$ rows of $\cM$ correspond to vectors $e_i^\star \in \tilde{M}$ dual to the divisors of $Z$ determined by the $k$ rays of $\bar{C}$. Let $b_i$ for $i \in [k]$ denote the ray generators of $C$. Fixing a $u \in \V{S}$ such that $v \in C_u$;
\begin{align*}
\langle \rho_s, e^\star_i\rangle  &= \langle \rho_s, \iota_u(b_i)\rangle \\
 &= \langle u(s), b_i \rangle,
\end{align*}
by Lemma~\ref{lem:projecting_struts}. Varying $s$ we see that these values are nothing other than the co-ordinates of the standard simplex $\Delta_{k-1}$, the image of $v^\star \to \RR^k$ defined by its description as a Cayley sum.

Conversely, assume that $\tilde{C}_v$ is unimodular for every vertex $v \in \V{P^\circ}$. It immediately follows from Proposition~\ref{pro:iota_union} that $P^\circ$ is cracked along $\Sigma$. To see that scaffolding must be full we note that for each strut $s = ((D,\chi)$ such that $u(s) \in v^\star$ the intersection $(P_D+\chi) \cap v^\star$ is contained in a single vertical face. If none of the inequalities appearing in Lemma~\ref{lem:normals} are redundant then it follows from the count of these inequalities made above that there must be a single strut covering each vertical face.

Note that the inequalities $\langle e_i, - \rangle \geq 0$ clearly cannot be redundant; so we check for redundancy among the inequalities $\langle \rho_s, -\rangle \geq -1$. Given a point $p \in C_u$ it follows from Lemma~\ref{lem:projecting_struts} that $\langle \rho_s,\iota(p)\rangle = \langle u(s), p\rangle$ for any $u \in \V{S}$. However, for some $u \in \V{S}$, $w := u(s)$ is a vertex of $v^\star$. Moreover, we have assumed that $s$ the only element of $S$ such that $w \in P_D+\chi$. The vertex $w$ defines a facet of $C_u$, and since $\langle w,-\rangle \geq -1$ is not redundant in defining $C_u$ the inequality $\langle \rho_s,-\rangle$ is not redundant in defining $\iota(C_u)$, and hence $\tilde{C}_v$.
\end{proof}
% !TEX root = paper.tex
%----------------------------------------------------------------------
\section{Smoothing complete intersections}
\label{sec:smoothing_cis}
%----------------------------------------------------------------------

We now restrict our attention to scaffoldings of polytopes cracked along the fan of a product of projective spaces. As shown in \cite{CKP17}, the embeddings obtained though this construction express $X_P$ as a complete intersection inside $Y_S$.

\begin{pro}
	\label{pro:ci_shape}
	Given a cracked polytope $P$ and a full scaffolding $S$ of $P$ with shape
	\[
		Z = \PP^{k_1}\times \cdots \times \PP^{k_r},
	\]
	the image of $X_P \hookrightarrow Y_S$ is a complete intersection.
\end{pro}
\begin{proof}
	The homogeneous co-ordinate ring of $Y_S$ is generated by variables $x_{i,j}$ for $i \in [r]$, $j \in [k_i]$, and variables $y_s$ for $s \in S$. Let $e_{i,j}$ denote the standard basis vector in $\Div_{T_{\bar{M}}}Z$ corresponding to $x_{i,j}$. The sublattice $\theta(N) \subset \tilde{N}$ is the intersection of the hyperplanes 
	\[
		H_i := \Big\{v \in \tilde{N} : \big\langle \sum_{j \in [k_i]}{e^\star_{i,j}}, v \big\rangle = 0\Big\},
	\]
	for $i \in [r]$. Thus $X_P$ is cut out by the equations 
	\[
		\prod_{j \in [k_i]}{x_{i,j}} = \prod_{s \in S}y_s^{-\sum_{j \in [k_i]}\langle \rho_s,e^\star_{i,j}\rangle}.
	\]
	It is easily seen that these equations generate the toric ideal of the image of $X_P$ in the affine chart defined by $\tilde{C}_v$, for any $v \in \V{P^\circ}$.
\end{proof}

We fix $Z$ for the remainder of this section to be the product of projective spaces $\PP^{k_1}\times \cdots \times \PP^{k_r}$. We let $\bar{\Sigma}$ be the fan determined by $Z$, and fix a splitting $N = N_U \oplus \bar{N}$. If $S$ is a scaffolding with shape $Z$, we let $L_i \in \Pic Y_S$ denote the line bundle $\cO_{Y_S}\{\prod_{j \in [k_i]}{x_{i,j}} = 0\}$.

\begin{rem}
	In fact, following \cite{CKP17}, scaffolding with shape $Z$ isomorphic to any toric tower of projective space bundles will express $X_P$ in $Y_S$ as a complete intersection. Our restriction to the product of projective spaces case is a simplifying one, but we expect results to hold true in the more general context. 
\end{rem}

We first introduce an important ingredient, the notion of a \emph{slab}. This terminology is taken from the work of Gross--Siebert~\cite{Gross--Siebert}, though our context and definition differ slightly. In particular our definition closer to that of \emph{naked slabs} since we do not yet decorate them with particular \emph{slab functions}.

\begin{dfn}
	Given a scaffolding $S$ of a Fano polytope $P$, we define the collection \emph{slabs} to be the collection of polytopes formed by intersecting codimension one cones of $\Sigma$ with $P^\circ$.
\end{dfn}

We index the slabs in $P^\circ$ by indexing the torus invariant curves in $Z$.
\begin{lem}
	\label{lem:curves_in_Z}
	The codimension one cones of $\Sigma$ are in bijection with functions 
	\[
		f\colon [r+1] \to \coprod_{i \in [r]}[k_i],
	\]
	such that $f(i) \in [k_i]$, and we define $i_\fs \in [r]$ to be the index such that $f(r+1) \in [k_{i_\fs}]$.
\end{lem}

Given a slab $\fs$, we denote the toric variety defined by its normal fan as $X_\fs$. Note that, as $P$ is cracked, $X_\fs$ is smooth.

\begin{lem}
	\label{lem:finding_slabs}
	Fix a Fano polytope $P$ and a slab $\fs = \sigma \cap P^\circ$, where $\sigma$ is a codimension one cone of $\Sigma$. Let $f$ be the function associated to $\fs$ by Lemma~\ref{lem:curves_in_Z}. $X_\fs$ is isomorphic to the toric stratum of $Y_S$ defined by the intersection of the vanishing loci $\{x_{i,f(i)} = 0\}$, for $i \in [r]$ and the vanishing locus $\{x_{i_\fs,f(r+1)}\}$.
\end{lem}
\begin{proof}
	This follows immediately from Lemma~\ref{lem:curves_in_Z} and the fact that 
	\[
	\iota(\fs) = \Ann\big(  \{e_{i,f(i)} : i \in [r]\} \cup\{e_{i_\fs,f(r+1)}\} \big) \cap Q_S.
	\]
\end{proof}

We let $\mu_\fs \colon X_\fs \to Y_S$ denote the inclusion corresponding to the inclusion of $\fs \hookrightarrow Q_S$. Each slab $\fs$ corresponds to a curve $C_\fs$ in $Z$. The index $i_\fs$ appearing in Lemma~\ref{lem:curves_in_Z} is the unique $i \in [r]$ such that $C_\fs$ projects to a curve under the map $Z \to \PP^{k_{i_\fs}}$.

\begin{dfn}
	\label{dfn:slab_divisor}
	Fix a scaffolding $S$ of a reflexive polytope $P$, and a slab $\fs \subset P^\circ$. Given a facet $\tau$ of $\fs$ contained in $\partial P^\circ$ we let $D_\tau$ denote the divisor of $X_\fs$ corresponding to the facet $\tau$, and we define
	\[
	a_\tau := 
	\begin{cases}
	\ell(\tau^\star) & \textrm{if $\tau$ is a face of $P^\circ$, and $\dim \tau^\star = 1$}\\
	0 & \textrm{otherwise.}
	\end{cases}
	\]
	where $\ell(e)$ denotes the lattice length of the edge $e$. We define the \emph{slab divisor}
	\[
	D_s := \sum_{\tau}{a_\tau D_{\tau}},
	\]
	where the sum is taken over facets $\tau$ of $\fs$ contained in $\partial P^\circ$. Moreover we let $L_\fs$ denote the line bundle $\cO(D_\fs) \in \Pic(X_\fs)$ for each slab $\fs$.
\end{dfn}

\begin{rem}
	Note that in Definition~\ref{dfn:slab_divisor} we assume that $P$ is reflexive. In fact this definition can be extended to all Fano polytopes, but becomes more complicated and we omit it here.
\end{rem}

Note that sections of $L_\fs$ are closely analogous to the \emph{slab functions} appearing in \cite{Gross--Siebert}. We relate the line bundles $L_\fs$ with the bundles $L_i$ defining the image of $X_P$ in $Y_S$.

\begin{lem}
	Fix a polytope $P$ cracked along the fan $\Sigma$, and $S$ a full scaffolding of $P$ with shape $Z$. Given a slab $\fs$, we have that $L_\fs = \mu_\fs^\star L_{i_\fs}$.
\end{lem}
\begin{proof}
	Let $\Sigma_\fs$ denote the fan determined by $X_\fs$. Observe that 
	\[
	\yy_\fs := \prod_{s \in S}y_s^{-\sum_{j \in [k_{i_\fs}]}\langle \rho_s,e^\star_{i_\fs,j}\rangle} \in \Gamma(Y_S,L_{i_\fs}).
	\]
	Moreover recall that there is a bijection between the facets $\tau$ of $\fs$ contained in $\partial P^\circ$ and the set
	\[
	\{s \in S : \langle \rho_s, - \rangle \geq -1 \text{ defines a facet of $\fs$} \}.
	\]
	Thus the pullback of $\{ \yy_\fs = 0\}$ to $X_\fs$ is the divisor obtained by labelling each such $\rho_s$ with the integer $\sum_{j \in [k_{i_\fs}]}\langle -\rho_s,e^\star_{i_\fs,j}\rangle$. We now interpret this integer as an edge length.
	
	Observe that $\Pic(Z) \cong \ZZ^r$, and that the map $\nu$ in the short exact sequence
	\[
		0 \to N \stackrel{\theta}{\to} \tilde{N} \stackrel{\nu}{\to} \Pic(Z) \to 0,
	\]
	sends $x \in \tilde{N}$ to $\sum_{i \in [r]}{a_i\bar{e}_i}$, where $\bar{e}_i$ is the pullback of the hyperplane class of $\PP^{k_i}$ to $Z$ along the canonical projection, and 
	\[
	a_i := \sum_{j \in [k_{i}]}\langle x,e^\star_{i,j}\rangle.
	\]
	Writing $x = (D,\chi) \in \Div_{T_{\bar{M}}}Z \oplus N_U$, the polytope $P_D \cong \prod_{i \in [r]}a_i\Delta_{k_i}$ is a product of dilated standard simplices. Note that each edge in the $i$th factor has length $a_i$. Thus the coefficient attached to each $\rho_s$ defining a facet $\tau$ of $\fs$ is precisely the edge length $\ell(\tau^\star)$ of $P_D$ corresponding to $C_\fs$, where $s = (D,\chi)$.
\end{proof}

\begin{dfn}
	\label{dfn:positive}
	Given a reflexive polytope $P$, cracked along a fan $\Sigma$, and a scaffolding $S$ of $P$, we say \emph{$S$ is positive} if the image of $\Gamma(Y_S,L_{i_\fs}) \to \Gamma(X_\fs,L_\fs)$ is a basepoint free linear system for all slabs $\fs$.
\end{dfn}

\begin{rem}
It would appear more natural, and closer to the definition appearing in \cite{Gross--Siebert}, to insist only that the complete linear system $\Gamma(X_\fs,L_\fs)$ is basepoint free. In fact in \cite{P:ZAffine} we show that this is precisely the condition needed to smooth $P^\circ$ as an integral affine manifold. Note however that this condition is independent of $S$, while Definition~\ref{dfn:positive} is not. We conjecture that if the linear system $\Gamma(X_\fs,L_\fs)$ is basepoint free for all $\fs$ then a smoothing of $X_P$ exists. Our slightly more restrictive definition guarantees the existence of an embedded smoothing in $Y_S$.
\end{rem}

\begin{conjecture}
	\label{con:smoothing}
	Fix a unimodular rational fan $\Sigma$ in $M_\RR$ such that $Z \in \TV(\bar{\Sigma})$ is projective, and a polytope $P$ such that $P^\circ$ is cracked along $\Sigma$. If $\Gamma(X_\fs,L_\fs)$ is basepoint free for all slabs $\fs$, then $X_P$ admits a smoothing.
\end{conjecture}

We study the equations defining $X_P \subset Y_S$ in the toric affine chart of $Y_S$ with cone dual to $\tilde{C}_v$ for a given vertex $v \in \V{P^\circ}$. Let $S(v) \subset S$ denote those $s \in S$ such that the inequality $\langle \rho_s,-\rangle \geq -1$ appears in Lemma~\ref{lem:normals} as a defining inequality for $\tilde{C}_v$. Let $B(v) = \coprod_{i \in [r]}B_i(v)$ denote the set of vectors $e_{i,j}$ dual to vectors $e^\star_{i,j}$ which are not contained in the minimal cone of $\Sigma$ containing $v$. In this notation, restricting the equations defining $X_P \subset Y_S$ to $\Spec\kk[(\tilde{C}_v-v)\cap\tilde{M}]$, we obtain the equations:

\begin{equation}
\label{eq:cutting_out}
	\prod_{j \in B_i(v)}{x_{i,j}} = \prod_{s \in S(v)}{y^{l_{i,s}}_s}~\textrm{for all $i \in [r]$}
\end{equation}
where 
\[
l_{i,s} := {-\sum_{j \in [k_i]}\langle \rho_s,e^\star_{i,j}\rangle}.
\]

We now prove the first part of Theorem~\ref{thm:affine_smoothing}, namely that vanishing locus of a general section of $\bigoplus_{i \in [r]} L_i$ is a smooth variety.

\begin{proof}[Proof of Theorem~\ref{thm:affine_smoothing} (\ref{it:smoothing})]
	Fix a point $x \in \im \Sing(X_P) \subset Y_S$ and let $p \in \partial P^\circ$ denote its moment map image. Let $C$ be the minimal cone of $\Sigma$ containing $p$. The cone $C$ corresponds to a toric stratum $X_C$ of $Z$, the product of toric strata $X_i$ of $\PP^{k_i}$ for $i \in [r]$. Let $v$ be a vertex of the minimal face of $P^\circ$ containing $X$, so that $x \in \Spec\kk[(\tilde{C}_v - v)\cap \tilde{M}]$. We  replace $Y_S$ by a resolution, or by the complement of $\Sing Y_S$; by Theorem~\ref{thm:smooth_ambient} this can be done in the complement of a Zariski neighbourhood of $x \in Y_S$.
	
	Let $V$ denote the variety cut out by the equations \eqref{eq:cutting_out} such that $X_i$ is zero dimensional. Since $B_i(v)$ is a singleton for each such equation, this equation becomes $x_{i,j} = \prod_{s \in S(v)}{y^{l_{i,s}}_s}$ for some $j \in [k_i]$. Thus the variety $V$ is smooth as $\tilde{C}_v$ is unimodular.
	
	If $X_i$ is not zero dimensional $x \in \im X_{\fs}$ (and $p \in \fs$) for all the slabs $\fs$ corresponding to curves of $Z$ contained in $X_i$. Consider the first value $i$ such that $|B_i(v)|>1$, then $L_i$ is basepoint free on $X_\fs$, and hence in a neighbourhood of $x \in Y_S$. It follows from a version of Bertini's theorem over $\kk$ that since $V$ is smooth and affine (and hence quasi-projective) the singularities of the vanishing locus of a general section of $L_i$ on $V$ are contained in the base locus of $L_i$. Removing the base locus of $L_i$ on $X$ and inductively applying this version of Bertini's theorem we obtain that the vanishing locus of a general section of $\bigoplus_{i \in [r]}{L_i}$ is smooth in a neighbourhood of $x$. Since $x$ was chosen arbitrarily from the singularities of $X_P$, and applying a standard compactness argument, we see that perturbing the equations \eqref{eq:cutting_out} smooth $X_P$ in $Y_S$.
\end{proof}

We now construct the divisor $D_S$ appearing in the statement of Theorem~\ref{thm:affine_smoothing}.

\begin{dfn}
	\label{dfn:anticanonical_divisor}
	Let $D_S$ be the divisor in $Y_S$ corresponding to function on the rays of the normal fan to $Q_S$ which sends each ray generated by $\rho_s$ for some $s \in S$ to $1$.
\end{dfn}

We can now conclude the proof of Theorem~\ref{thm:affine_smoothing} by describing the restriction of the divisor $D_S$ to a smoothing of $X_P$. The strategy is broadly the same: for each equation in \eqref{eq:cutting_out} we can either eliminate it, or make a transversality argument based on the positivity of $S$.

\begin{proof}[Proof of Theorem~\ref{thm:affine_smoothing} (\ref{it:divisor})]
	Let $X$ be a general smoothing of $X_P$ obtained by perturbing equations \eqref{eq:cutting_out}. Recall that $\prod_{j \in [k_i]}{x_{i,j}} \in \Gamma(Y_S,L_i)$ for each $i\in [r]$. Thus the restriction of $D_S$ to $X$ (or $X_P$) is anti-canonical by the adjunction formula.
	
	To show that the restriction of $D_S$ to $X$ is simple normal crossings (snc) we must show that each component of $D_S|_X$ is smooth, and $D_S|_X$ is locally of the form $\prod_{i \in [a]}{x_i} = 0$, where $x_i$ are local parameters on $Y_S$ and $a \in [\dim N]$. The argument used to show that $X_P$ smooths inside $Y_S$ extends immediately to each component $\{y_s = 0\} \cap X_P$, ensuring each component of $D_S|_X$ is smooth.
	
	Fix a point $x \in X_P$ such that the image of $x$ in $Y_S$ lies at the intersection of divisors $\{y_s = 0\} \cap \im X_P$ for $s \in S' \subseteq S$, where $S'$ is some subset of $S$. Let $p$ be point in $P^\circ$ corresponding to $x$, and let $v \in \V{P^\circ}$ be a vertex of the minimal face of $P^\circ$ containing $p$. As above the local affine piece, 
	\[
	Y_v := \Spec\kk[(\tilde{C}_v-v) \cap \tilde{M}],
	\]
	has co-ordinates $y_s$ for $s \in S(v)$ -- noting that $S' \subseteq S(v) \subseteq S$ -- and $x_{i,j}$ where $e_{i,j} \in B(v)$. The restriction of the divisor $D_S$ to $Y_v$ is given by the equation $\prod_{s \in S(v)}{y_s = 0}$. This divisor remains snc after intersection with the $i$th equation \eqref{eq:cutting_out} if $|B_i(v)| = 1$ (with co-ordinates on the vanishing locus given by eliminating $x_{i,j}$, where $B_i(v) = \{e_{i,j}\}$). As above, if $|B_i(v)| > 1$, the base locus of $L_i$ is disjoint from a neighbourhood of $x$. Thus, inductively applying the fact that the sum of an snc divisor and a general member of a free linear system is snc -- see \cite[Lemma~$9.1.9$]{Laz04} -- the result follows.
\end{proof}
% !TEX root = paper.tex
%----------------------------------------------------------------------
\section{Examples}
%----------------------------------------------------------------------

We consider examples in dimension $3$, using the famous classifications of reflexive $3$-polytopes by Kreuzer--Skarke~\cite{KS98}, and of three dimensional Fano varieties by Mori--Mukai~\cite{Mori--Mukai:Manuscripta,Mori--Mukai:Kinosaki}.

\begin{eg}
	We construct a Fano variety in the family \MM{2}{18} in the Mori--Mukai list from a cracked polytope. The complete intersection description we obtain appears in \cite{CCGK}. Let $N := \ZZ^3$, $N_U := \langle e_3\rangle$, and $\bar{N} := \ZZ^2 = \langle e_1,e_2\rangle$. We let $Z := \PP^2$, and let $\Sigma$ denote the corresponding fan.
	
	The left hand image in Figure~\ref{fig:eg2-18} shows the scaffolding of $P$ using $3$ struts: a pair of triangles labelled $s_1$ and $s_2$, and the remaining vertex. The right hand image shows $P^\circ$ cracked into three Cayley polytopes, two of which ($Q_2$ and $Q_3$) are Cayley sums of a pair of triangles, while $Q_1$ is the Cayley sum of three line segments. Note that $X_P$ is a hypersurface in $Y_S$, the vanishing locus of a section of some $L \in \Pic Y_S$.
	
	We check positivity along the slab $\fs = Q_1 \cap Q_3$. Observe that $X_\fs \cong \FF_1$, and let $\pi \colon \FF_1 \to \PP^2$ be the contraction of the $(-1)$-curve. We have that $L_\fs = \pi^\star{\cO_{\PP^2}(2)}$, which is a basepoint free divisor and, verify that in this example the map $\Gamma(Y_S,L) \to \Gamma(X_\fs,L_\fs)$ is surjective.
	
	The toric variety $Y_S$ admits a morphism to $\PP^2\times \PP^1$, realising a smoothing $X$ of $X_P$ as a double cover of $\PP^2 \times \PP^1$ branched in a divisor of bidegree $(2,2)$; see \cite{CCGK} for more details.
	
		\begin{figure}
		\centering
		\begin{minipage}[b]{0.49\textwidth}
			\includegraphics[width=\textwidth,  trim = 10mm 10mm 10mm 10mm, clip]{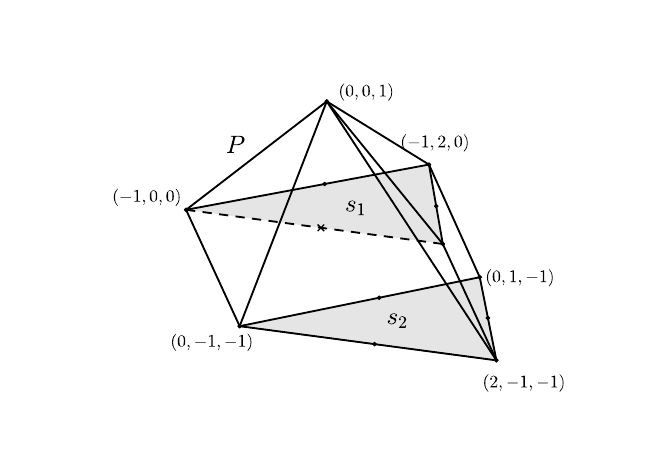}
		\end{minipage}
		\hfill
		\begin{minipage}[b]{0.49\textwidth}
			\includegraphics[width=\textwidth, trim = 15mm 7mm 15mm 8mm, clip ]{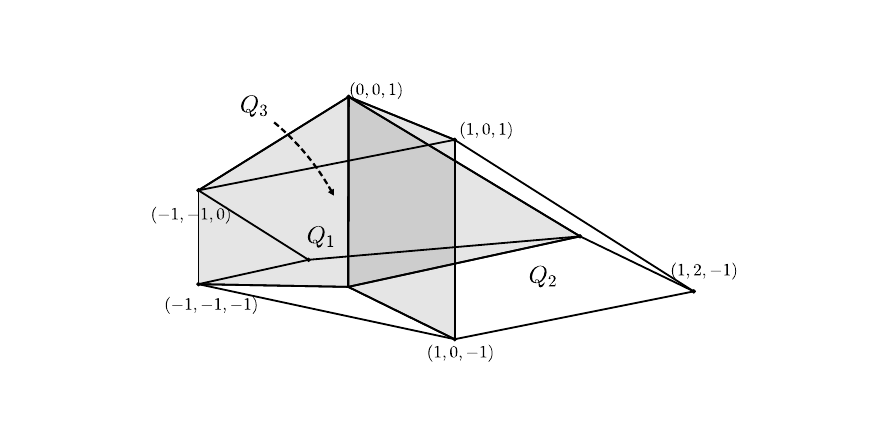}
		\end{minipage}
		\caption{Construction of a variety in the family \MM{2}{18} via a cracked polytope.}
		\label{fig:eg2-18}
	\end{figure}
\end{eg}

\begin{eg}
	We describe an important class of non-examples. Consider the reflexive polytope $P$ with PALP id $15$, whose corresponding toric variety $X_P$ is shown to be non-smoothable by Petracci in \cite{P18}. A neighbourhood of $\Sing X_P$ is isomorphic to a bundle of $A_1$ (surface) singularities over $\PP^1$. This is represented on $P$ by an edge of length $2$ which has direction vector $u := (0,1,1)$. $P^\circ$ is cracked along the fan $\Sigma$ consisting of two half spaces meeting along $M_U = \Ann(u) \subset M_\RR$. As observed in Remark~\ref{rem:P1-shape} this polytope admits a full scaffolding with shape $\PP^1$. However this scaffolding is not positive.

	\begin{figure}
		\includegraphics[scale = 0.5]{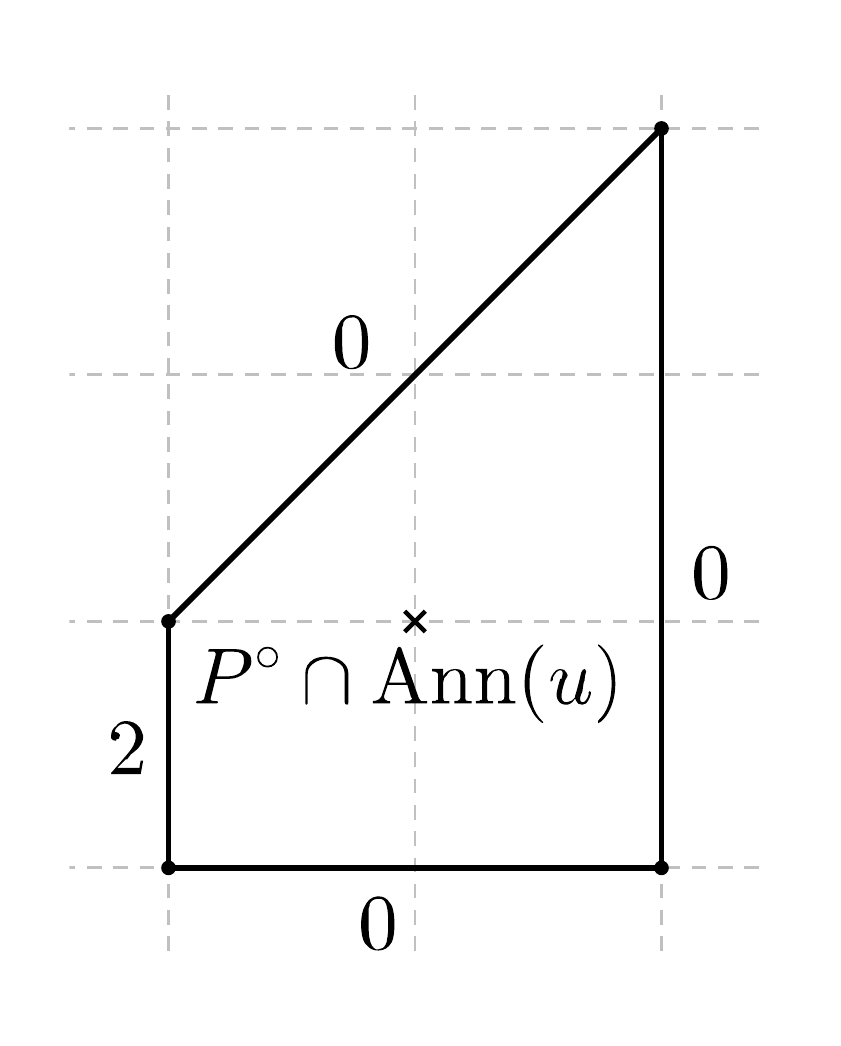}
		\caption{Slab divisor for a non-positive scaffolding.}
		\label{fig:non_eg}
	\end{figure}

	The unique slab $\fs = \Ann(u) \cap P^\circ$ is shown in Figure~\ref{fig:non_eg}, together with the divisor $D_\fs$. The toric variety $X_\fs \cong \FF_1$ with $(-1)$-curve $E$, and $D_\fs = 2E$; hence $S$ cannot be positive. The same analysis applies to the reflexive polytopes with PALP ids in the set (a subset of the list appearing in \cite{P18}),
	\[
	\{16,58,59,61,65,66,192,193,197\}.
	\]
	That is, the corresponding (dual) polytopes are cracked along the hyperplane dual to the direction $u$ of an edge defining a transverse $A_n$ singularity. However no full scaffolding of these polytopes is positive, as $L_\fs$ -- where $\fs = \Ann(u) \cap P^\circ$ -- is a multiple of a curve with negative self intersection.
\end{eg}

%-------------------------------------------------------------------------------
\bibliographystyle{plain}
\bibliography{bibliography}
%-------------------------------------------------------------------------------
\end{document}